\documentclass[12pt,reqno]{amsart}
\usepackage{fullpage}
\usepackage[top=1.5in, bottom=2in, left=2.5in, right=2.5in]{geometry}
\pdfoutput=1

\usepackage{amsmath,amsthm,amsfonts,amssymb,ifpdf,bm}
\usepackage{amssymb}
\usepackage{amsmath}
\usepackage{amsthm}
\usepackage{graphicx}
\usepackage[all]{xy}
\usepackage{enumerate}
\usepackage{tikz-cd}

\usepackage[utf8]{inputenc}
\usepackage[T1]{fontenc}
\usepackage{mathtools}
\usepackage{lipsum}
\usepackage{enumerate}
\usepackage{setspace}
\usepackage{pst-node}
\usepackage{blindtext}
\usepackage{upgreek}
\usepackage{dsfont}
\usepackage{array}
\usepackage{booktabs}
\usepackage{multirow}
\usepackage{pdflscape}
\usepackage{amsmath}
\usepackage{mathrsfs}
\theoremstyle{plain} 
\newtheorem{theorem}{Theorem}[section]

\newtheorem{lemma}[theorem]{Lemma}
\newtheorem{corollary}[theorem]{Corollary}

\numberwithin{equation}{section}

\theoremstyle{definition} \newtheorem{definition}[theorem]{Definition}
\newtheorem*{definition*}{Definition}

\theoremstyle{plain}

\newtheorem{remark}{Remark}

\usepackage[all]{xy}

\ifpdf
  \usepackage[draft]{hyperref} 
 
 \usepackage{thumbpdf}
\else
  \usepackage[draft]{hyperref}
\fi
\newif\ifdiscs
\discstrue
\newcommand{\ncom}{\newcommand}
\ncom{\mylabel}[1]{{\rm (#1)}\label{#1}}
\ncom{\Hom}{{\textit{Hom}}}
\ncom{\eop}{{\hfill $\Box$}}
\usepackage{enumitem,longtable}
\setlist[enumerate]{leftmargin=*,label={\rm(\arabic*)}}

\makeatletter
\@namedef{subjclassname@2020}{%
	\textup{2020} Mathematics Subject Classification}
\makeatother

\title[Strict Log-concavity of $k$-coloured Partitions]{Strict Log-concavity of $k$-coloured Partitions}
\author{Kathrin Bringmann}
\address{Department of Mathematics and Computer Science, Division of Mathematics, University of Cologne, Weyertal 86--90, 50931 Cologne, Germany}
\email{kbringma@math.uni-koeln.de}
\author{Ben Kane}
\address{Department of Mathematics, The University of Hong Kong, Pokfulam, Hong Kong}
\email{bkane@hku.hk}
\author{Anubhab Pahari}
\address{Department of Mathematics, IIT Madras, Chennai, 600036, India.}
\email{ma22d012@smail.iitm.ac.in, anubhabpahari@gmail.com}
\author{Larry Rolen}
\address{Department of Mathematics,
1420 Stevenson Center,
Vanderbilt University,
Nashville, TN 37240, USA}
\email{larry.rolen@vanderbilt.edu}

\begin{document}
\newcommand{\R}{{\mathbb R}}
\newcommand{\C}{{\mathbb C}}
\newcommand{\N}{{\mathbb N}}
\newcommand{\Q}{{\mathbb Q}}
\newcommand{\Z}{{\mathbb Z}}
\newcommand{\K}{{\mathbb K}}
\newcommand{\D}{{\mathbb D}}
\def\Hilb{{\rm Hilb}}
\def\dim{{\rm dim}}
\def\Quot{{\rm Quot}}
\def\min{{\rm min}}
\def\card{{\rm card}}
%\uspackage{auto-pst-pdf}
\newcommand{\disjointunion}{\amalg}
% \usepackage[
%   separate-uncertainty = true,
%   multi-part-units = repeat
% ]{siunitx}

%\input{Body.tex}

\begin{abstract}
	In recent years, there has been extensive work on inequalities among partition functions. In particular, Nicolas, and independently DeSalvo--Pak, proved that the partition function $p(n)$ is eventually log-concave. Inspired by this and other results, 
    Chern--Fu--Tang first conjectured log-concavity of $k$-coloured partitions. Three of the authors and Tripp later proved this conjecture by introducing recursive sequences and a strict inequality for fractional partition functions, giving explicit errors. In this paper, we show that the log-concavity is, in fact, strict for $k\geq 2$. We shed further light on this phenomenon by utilizing Hardy--Littlewood--P\'olya's notion of majorizing.
	We prove that for partitions $\bm{a},\bm{b}$ of $n\in\N$, if $\bm b$ majorizes $\bm a$, then $p_k(\bm{a})>p_k(\bm{b})$. Numerical calculations indicate that our result is sharp.
\end{abstract}

\maketitle

\section{Introduction and Statement of Results}\label{sec:Introduction}
A \emph{partition} of $n \in \mathbb{N}$ is an $r$-tuple $(n_1,n_2,\cdots,n_r)\in \mathbb{N}^r$ with $r\le n$ such that $ n_1\geq n_2\geq \cdots\geq n_r\geq 1$ and $\sum_{j=1}^rn_j=n$. We call $n_j$ the \emph{$j$-th part} of the partition and denote by $p(n)$ the number of partitions of $n$. By Euler, the generating function for $p(n)$ is given by 
\begin{equation*}
    \sum_{n\ge0} p(n)q^n=\prod_{n\ge1} \frac{1}{1-q^n}.
\end{equation*}
 More generally, a \emph{$k$-coloured partition} of $n$ is a partition of $n$ where each part is assigned one of $k$ colours, and $p_k(n)$ counts the number of $k$-coloured partitions of $n$. Note that $p_1(n)=p(n)$. The generating function of $p_k(n)$ is given by
\begin{equation*}%\label{gen fn for p_k(n)}
    \sum_{n\ge0}p_k(n)q^n=\prod_{n\ge1}\frac{1}{(1-q^n)^k},
\end{equation*}
where we set $p_k(0):=1$. This identity makes $p_k(n)$ important in combinatorics and in other fields of mathematics. One such application can be found in algebraic geometry.  For a smooth, projective surface $S$, there is a smooth projective variety of dimension $2n$, called the Hilbert scheme of $n$ points over a surface, denoted by $S^{[n]}$. We denote the topological Euler characteristic of $S$ and $S^{[n]}$ by $\chi(S)$ and $\chi(S^{[n]})$, respectively. The generating function of $\chi(S^{[n]})$ is given by (see \cite[equation $(2)$, Theorem 0.1]{G\"ottsche'90}),
\begin{equation*}
    \sum_{n\ge0}\chi\left(S^{[n]}\right)q^n=\prod_{n\ge1}\frac{1}{(1-q^n)^{\chi(S)}}.
\end{equation*}
Thus we have
\begin{equation*}
    \chi\left(S^{[n]}\right)=p_{\chi(S)}(n).
\end{equation*}
To give another example, in their study of supersymmetric gauge theory and random partitions, 
Nekrasov and Okounkov \cite{NO} proved the famous formula
\[
\sum_{\lambda\in\mathcal P}q^{|\lambda|}\prod_{h\in\mathcal H(\lambda)}\left(1-\frac{\alpha}{h^2}\right)=\prod_{n\geq1}(1-q^n)^{\alpha-1}, \quad \alpha\in\C.
\]
Here, $\mathcal P$ denotes the set of all partitions, $|\lambda|$ denotes the size of the partition, and $\mathcal H(\lambda)$ denotes the (multi)set of hook lengths of $\lambda$.
Key combinatorial applications of this formula, as well as extensions and an elementary proof, were given by Han \cite{Han}.

Inspired by Bessenrodt--Ono's proof \cite{BO} that $p(a)p(b) > p(a + b)$ for $a,b\in\N_{>1}$ with $a+b>9$, Chern--Fu--Tang \cite{MR3826751}
showed that the same inequality holds for $p_k(n)$ with finitely many exceptions.
They then speculated a strengthening of this result, inspired by the log-concavity results of Nicolas \cite{Nicolas} and DeSalvo--Pak \cite{DeSalvoPak}.
Specifically, they conjectured that for $1\le\ell<n$ and $k\geq 2$, except for $(k,n,\ell)=(2,6,4)$, we have
\begin{equation*}
    p_k(\ell+1)p_k(n-1)\geq p_k(\ell)p_k(n).
\end{equation*}
Three of the authors and Tripp \cite{MR4287510} proved this conjecture. In this paper, we modify and strengthen this conjecture.
\begin{theorem}\label{st. log con for length 2}
     For $k,n\in\mathbb N$ with $k\geq3$ and $(k,n)\neq(3,1)$, $p_k$ is strictly log-concave.  That is, we have
    \begin{equation*}
        p_k^2(n)>p_k(n-1)p_k(n+1).
    \end{equation*}
    
\end{theorem}
We deduce the following.
\begin{corollary}\label{st. log con for ln 2, equv}\hspace{0cm}\\
    {\rm(1)} For $n,\ell \in \mathbb{N}_0$ with $0\le \ell\le n-2$ and $k\in \mathbb{N}_{\ge4}$, we have 
    \begin{equation*}
        p_k(n-1)p_k(\ell+1)> p_k(n)p_k(\ell).
    \end{equation*}
    {\rm(2)} For $k=3$ and $1\le \ell\le n-2$, we have 
    \begin{equation*}
        p_3(n-1)p_3(\ell+1)> p_3(n)p_3(\ell).
    \end{equation*}
 	{\rm(3)} For $k=2$ and $1\le \ell \le n-2$ with $(n,\ell)\neq (6,4)$, we have 
    \begin{equation*}
        p_2(n-1)p_2(\ell+1)> p_2(n)p_2(\ell).
    \end{equation*}
\end{corollary}
Next, we prove a result which extends Corollary \ref{st. log con for ln 2, equv} to partitions of general length. For this, we require the following definition. 
\begin{definition*}%\label{Defn_colour partition respects majorising order}
 For a sequence ${\bm a}=(a_1,a_2,\cdots,a_r)\in \mathbb{N}^r$, we define 
 \begin{equation}\label{pv}
     p_k({\bm a}):=p_k(a_1)p_k(a_2)\cdots p_k(a_r).
 \end{equation}
\end{definition*}
In this notation, Corollary \ref{st. log con for ln 2, equv} states (under the conditions of Corollary \ref{st. log con for ln 2, equv})
\begin{equation}\label{eqn:corinequality}
p_k(n-1,\ell+1)>p_k(n,\ell).
\end{equation}
Noting that $(n-1, \ell+1)$ and $(n, \ell)$ are both partitions of the same length, we next recall a partial ordering on partitions of the same length, known as ``majorization''. Related notions were considered earlier by Muirhead \cite{doi:10.1017/S001309150003460X}, Lorenz \cite{Lorenz}, Dalton \cite{Dalton}, Schur \cite{Schur}, and others. The following particular notation and terminology was given by Hardy, Littlewood, and P\'olya \cite{hardy}, \cite[Subsection 2.18]{hardy1952inequalities}.
\begin{definition}\label{Def: majorisation}
    For partitions ${\bm a}=(a_1,\cdots,a_r)$ and ${\bm b}=(b_1,\cdots,b_r)$ of $n\in\N_0$, we say ${\bm b}$ \textit{majorizes} ${\bm a}$, denoted by ${\bm b}\succeq{\bm a}$, if
    \begin{equation*}
            \sum_{j=1}^\nu a_j\leq \sum_{j=1}^\nu b_j, \text{ for all } \nu \in \{1,2,\cdots,r-1\}. 
    \end{equation*}
If $\bm{b}$ majorizes $\bm{a}$ and $\bm{a}\neq \bm{b}$, then we write $\bm{b}\succ \bm{a}$. 
\end{definition}
Note that the partition $(n, \ell)$ majorizes the partition $(n-1, \ell+1)$, and \eqref{eqn:corinequality} implies that $p_k$ reverses this partial ordering. In our next result, we extend the inequality in \eqref{eqn:corinequality} to other majorizations.
\begin{theorem}\label{thm:general}
Let ${\bm a}=(a_1,a_2,\cdots,a_r)$ and ${\bm b}=(b_1,b_2,\cdots,b_r)$ be partitions of $n\in\N$ with ${\bm b}\succ {\bm a}$. Then we have, for $k\ge 3$, 
    \begin{equation}\label{eqn: gen. ineq.}
      p_k({\bm a}) > p_k({\bm b}).
    \end{equation}
\end{theorem}
\begin{remark}
The proof of Theorem \ref{thm:general} requires a careful blending of analytic and combinatorial techniques. We use combinatorial properties of majorized partitions to reduce Theorem \ref{thm:general} to a proof of Corollary \ref{st. log con for ln 2, equv}. However, these combinatorial techniques alone do not seem to be sufficient to show Corollary \ref{st. log con for ln 2, equv} directly. Instead, Corollary \ref{st. log con for ln 2, equv} follows by asymptotic arguments from \cite{MR4287510}. As noted above, Corollary \ref{st. log con for ln 2, equv} is a special case of Theorem \ref{thm:general}. This implies that Theorem \ref{thm:general} extends Corollary \ref{st. log con for ln 2, equv}. 
\end{remark}

The paper is organized as follows. In Section \ref{sec:ProofThm4} we prove Theorem \ref{st. log con for length 2} and Corollary \ref{st. log con for ln 2, equv}. Section \ref{sec:Inequality} is devoted the proof of Theorem \ref{thm:general}. We finish this paper with questions for future research in Section \ref{sec:furtherquestions}.

\section*{Acknowledgments}
The first author has received funding from the European Research Council (ERC) under the European Union's Horizon 2020 research and innovation programme (grant agreement No. 101001179).\hspace{.5em}The research of the second author was supported by grants from the Research Grants Council of the Hong Kong SAR, China (project numbers HKU 17314122, HKU 17305923). The third author has received financial support under the scheme, the Prime Minister's Research Fellowship, Govt. of India(PMRF, ID: 2503482).
The fourth author’s work was supported by a grant from the Simons Foundation (853830, LR).

\section{Proofs of Theorem \ref{st. log con for length 2} and Corollary \ref{st. log con for ln 2, equv}}\label{sec:ProofThm4}
We first define convolutions of sequences.
\begin{definition}\label{def: conv}
    Let $\{a_n\}_{n= 0}^N$ and $\{b_n\}_{n= 0}^M$ be sequences of positive integers. Their \emph{convolution}\footnote{This is also sometimes called the Cauchy product.} is the sequence $\{e_n\}_{n=0}^{M+N}$ ($N,M\in \N$) with
    \begin{equation*}
        e_n:=\sum_{j=0}^na_j b_{n-j}.
    \end{equation*}
\end{definition}

The following lemma, which is direct to show, gives the convolution of $p_{k_1}$ and $p_{k_2}$.
\begin{lemma}\label{lem:convcolour}
    The convolution of $p_{k_1}$ and $p_{k_2}$ is $p_{k_1+k_2}$.
\end{lemma}
The following lemma strengthens an equivalent definition of log-concavity of a sequence of positive integers (originally proven for a sequence of non-negative integers) from \cite{MR926131}. The proof is straightforward, so we omit it here.
\begin{lemma}\label{eqv reln for st log conc.}
    Let $\{a_n\}_{n=0}^N$ be a sequence of positive integers and $r\in\N$. Then the strict inequality
    \begin{equation}\label{in1}
        a_{\ell+1} a_{m-1}>a_{\ell} a_{m}
    \end{equation}
    for all $ r \leq \ell+1 < m \leq N$ is equivalent to the strict inequality,
     \begin{equation}\label{in2}
         a_n^2>a_{n-1}a_{n+1}.
     \end{equation}
     for all $r\leq n\leq N-1$.
\end{lemma}
\begin{proof}
    First assume that \eqref{in1} holds. Then \eqref{in2} follows directly by taking $\ell=n-1$ and $m=n+1$.  Note that, for $1\le r\le \ell+1<m\le N$, we have $n=\ell+1\ge r$ and $n\le N-1$. 
    
    For the other direction, define $\lambda_n:= \frac{a_n}{a_{n-1}}.$ From \eqref{in2}, we have, for $r\le n \le N$,
	\begin{equation*}
		\lambda_n=\frac{a_n}{a_{n-1}}>\frac{a_{n+1}}{a_n}=\lambda_{n+1}.
	\end{equation*}
	 Thus, the sequence $\{\lambda_n\}_{j=r}^{s}$ is strictly decreasing. If we take $r\le \ell+1<m\le N$, then we have $\lambda_{\ell+1}>\lambda_{m}$, i.e.,
	\begin{equation*}
		\frac{a_{\ell+1}}{a_{\ell}}>\frac{a_{m}}{a_{m-1}}.
	\end{equation*}
	So \eqref{in1} holds.
\end{proof}
\begin{remark}
	If $r=1$ in Lemma \ref{eqv reln for st log conc.}, then $a_n$ is called strictly log-concave.
\end{remark} 
The following lemma eases our work. The proof is inspired from \cite[Section 2]{Engel} (see also, \cite[Theorem 1.2]{WANG2008395}).
\begin{lemma}\label{convolution lemma}
    The convolution of two strictly log-concave sequences is strictly log-concave. 
\end{lemma}
\begin{proof}
     Let $\{a_j\}_{j=0}^N$ and $\{b_j\}_{j=0}^M$ be two strictly log-concave sequences. This means by definition, 
\begin{equation*}
	a_j^2>a_{j-1}a_{j+1} \text{ for } 1\le j \le N-1,\quad b_j^2>b_{j-1}b_{j+1} \text{ for } 1\le j \le M-1.
\end{equation*}
   So, by Lemma \ref{eqv reln for st log conc.}, we have
\begin{equation}\label{eq: st log eqv 1}
    \begin{split}
         a_{\ell+1}a_{m-1}&>a_\ell a_m \text{ for } 1 \le \ell+1 < m \le N,\\
        b_{\ell+1}b_{m-1}&>b_\ell b_m \text{ for } 1 \le \ell+1 < m \le M.
    \end{split}
\end{equation}
We next let $m=\ell+1$ and $\ell= m-1$, to get
\begin{equation}\label{eq: st log eqv 2}
    \begin{split}
         a_\ell a_m&>a_{\ell+1}a_{m-1} \text{ for } 1 \le m < \ell+1 \le N,\\
         b_\ell b_m&>b_{\ell+1}b_{m-1} \text{ for } 1 \le m < \ell+1 \le M.
    \end{split}
\end{equation}
   
    Following the definition of convolution, we let
    \begin{equation}\label{eq: conv}
        e_k \coloneqq \sum_{m = 0}^k a_m b_{k-m} \text{ for } 0\le k\le N+M.
    \end{equation}
    To prove that $\{e_k\}_{k=0}^{N+M}$ is strictly log-concave, we need to show that
    \begin{equation*}
        e_k^2> e_{k-1}e_{k+1} \text{ for } 1\le k \le N+M-1.
    \end{equation*}
    Plugging this into \eqref{eq: conv}, we need to show that
    \begin{equation}\label{eq: st log concav for conv}
        e_k^2=\sum_{0 \le \ell,m \le k} a_\ell b_{k-\ell}a_mb_{k-m} > \sum_{\substack{0\le \ell \le k-1\\0\le m\le k+1}} a_\ell b_{k-\ell-1}a_mb_{k-m+1}=e_{k-1}e_{k+1}.
    \end{equation}
    
    We first assume that $m > \ell-1$.
    Note that, 
    \begin{equation*}
        m>\ell+1  \Rightarrow k- m +1< k-\ell .
    \end{equation*}
    So if $m>\ell+1$, then, by \eqref{eq: st log eqv 1}, we have, 
    \begin{equation*}
        a_\ell a_m-a_{\ell+1}a_{m-1}<0 \text{ and } b_{k-\ell}b_{k-m}-b_{k-\ell-1}b_{k-m+1}<0.
    \end{equation*}
   This implies that
   \begin{equation*}
       (a_\ell a_m -a_{\ell+1}a_{m-1})(b_{k-\ell}b_{k-m}-b_{k-\ell-1}b_{k-m+1})>0.
   \end{equation*}
   
   Next assume that $m<\ell+1$. Note that,
   \begin{equation*}
       m<\ell+1  \Rightarrow k- m +1> k-\ell.
   \end{equation*}
       If $m<\ell+1$, then, by \eqref{eq: st log eqv 2}, we have
\begin{equation*}
        a_\ell a_m-a_{\ell+1}a_{m-1}>0 \text{ and } b_{k-\ell}b_{k-m}-b_{k-\ell-1}b_{k-m+1}>0.
    \end{equation*}
    Again, this implies that
   \begin{equation*}
       (a_\ell a_m-a_{\ell+1}a_{m-1})(b_{k-\ell}b_{k-m}-b_{k-\ell-1}b_{k-m+1})>0.
   \end{equation*}
   Thus if $m\ne \ell+1$, then we have  
    \begin{equation}\label{eq: gen ineq for conv}
        \begin{split}
            (a_\ell a_m-a_{\ell+1}a_{m-1})(b_{k-\ell}b_{k-m}-b_{k-\ell-1}b_{k-m+1})>0.
        \end{split}
    \end{equation}
    Next note that if $m=\ell+1$, then we have
    \begin{equation*}
        (a_\ell a_m-a_{\ell+1}a_{m-1})(b_{k-\ell}b_{k-m}-b_{k-\ell-1}b_{k-m+1})=0.
    \end{equation*}
    Note that by our convention, we let $a_k=b_k=0$ for $k<0$.
    
    Now, if we sum \eqref{eq: gen ineq for conv} over $-1\le \ell,m\le k+1$ on both sides, we get
    \begin{align}\label{eq: summation over ineq}
            & \sum_{-1 \le \ell,m \le k+1} (a_\ell a_m-a_{\ell+1}a_{m-1})(b_{k-\ell}b_{k-m}-b_{k-\ell-1}b_{k-m+1})>0 \nonumber\\
            &\Rightarrow \sum_{-1 \le \ell,m \le k+1} \left(a_\ell b_{k-\ell}a_mb_{k-m}+a_{\ell+1}b_{k-\ell-1} a_{m-1}b_{k-m+1}\right) \nonumber\\ 
            &\hspace{2.5cm} > \sum_{-1 \le \ell,m \le k+1} \left(a_\ell b_{k-\ell-1} a_mb_{k-m+1}+a_{\ell+1}b_{k-\ell}a_{m-1}b_{k-m}\right).
    \end{align}
     We first rewrite the left-hand side as
    \begin{flalign}\label{eq: LHS}
    	& \sum_{-1 \le \ell,m \le k+1}\left(a_\ell b_{k-\ell}a_mb_{k-m}+a_{\ell+1}b_{k-\ell-1} a_{m-1}b_{k-m+1}\right)\nonumber && \\ 
    	&\hspace{3cm}= \sum_{0 \le \ell,m \le k} a_\ell b_{k-\ell}a_mb_{k-m}+\sum_{\substack{-1\le\ell\le k-1\\1\le m\le k+1}} a_{\ell+1}b_{k-\ell-1} a_{m-1}b_{k-m+1}\nonumber && \\ 
    	&\hspace{3cm}= 2 \sum_{0 \le \ell,m \le k} a_\ell b_{k-\ell}a_mb_{k-m}  &&
    \end{flalign}
    by replacing $m\mapsto m+1$ and $\ell \mapsto \ell-1$ in the second sum.

    We next rewrite the right-hand side of \eqref{eq: summation over ineq}. We have
    \begin{align}\label{eq: RHS}
             \sum_{-1 \le \ell,m \le k+1} &\left(a_\ell b_{k-\ell-1}  a_mb_{k-m+1}+a_{\ell+1}b_{k-\ell}a_{m-1}b_{k-m}\right) \nonumber\\ &\hspace{2cm}= \sum_{\substack{0\le \ell \le k-1\\ 0\le m \le k+1}} a_\ell b_{k-\ell-1} a_mb_{k-m+1}+\sum_{\substack{-1\le\ell\le k\\1\le m\le k}} a_{\ell+1}b_{k-\ell}a_{m-1}b_{k-m} \nonumber\\ &\hspace{2cm}= 2 \sum_{\substack{0\le \ell \le k-1\\0\le m\le k+1}} a_\ell b_{k-\ell-1} a_mb_{k-m+1},
    \end{align}
    by replacing $m\mapsto \ell+1$ and $\ell\mapsto m-1$ in the second sum. Using \eqref{eq: summation over ineq}, \eqref{eq: LHS}, and \eqref{eq: RHS}, we obtain \eqref{eq: st log concav for conv}.
\end{proof}

From this we conclude following lemma. 
\begin{lemma}\label{reduce}
	Theorem \ref{st. log con for length 2} holds if it holds for $k\in\{3,4,5\}$ $(n>1~\text{if}~k=3)$.
\end{lemma}

\begin{proof}
	Note that one can write $k\in \mathbb{N}_{\geq 3}$ as
	\begin{equation*}%\label{eqn:coloursum}
	k=3j_1+4j_2+5j_3,
	\end{equation*}
	with $j_1,j_2,j_3 \in \mathbb{N}_{0}$. Hence, by Lemma \ref{convolution lemma} and Lemma \ref{lem:convcolour}, in order to prove Theorem \ref{st. log con for length 2}, it is enough to check the strict log-concavity of $p_k(n)$ for $k\in \{3,4,5\}$, unless $n=1$. For $n=1$, the claim follows directly for $k\neq3$ by computing $p_k(0)=1$, $p_k(1)=k$, and $p_k(2)=\frac{k(k+3)}{2}$. 
\end{proof}

We next require the following theorem, which was proven by three of the authors and Tripp \cite[Theorem 1.2]{MR4287510}. Throughout, the notation $f(x)=O_{\leq}(g(x))$ means that $|f(x)|\leq g(x)$ for a positive function $g$ and for all $x$ in the domain in which $f$ and $g$ are defined.
\begin{theorem}\label{thm of BKRT}
    Let $\alpha \in \mathbb{R}_{\geq 2}$ and $n,\ell \in \mathbb{N}_{\geq 2}$ with $n>\ell +1$. Set $N:=n-1-\frac{\alpha}{24}$, $L:=\ell-\frac{\alpha}{24}$, and suppose that $L \geq \max\{2\alpha^{11},\frac{100}{\alpha-24}\}$. Then we have
     \begin{multline*}%\label{error term}
            p_{\alpha}(n-1)p_{\alpha}(\ell+1)-p_{\alpha}(n)p_{\alpha}(\ell)\\=\pi \left(\frac{\alpha}{24}\right)^{\frac{\alpha}{2}+1}(NL)^{-\frac{\alpha+5}{4}} \left(\sqrt{N}-\sqrt{L}\right)e^{\pi \sqrt{\frac{2\alpha}{3}}\left(\sqrt{N}+\sqrt{L}\right)}\left(1+O_{\leq}\left(\frac{14}{15}\right)\right).
    \end{multline*}
\end{theorem}
We now use Theorem \ref{thm of BKRT} and a computer calculation to show Theorem \ref{st. log con for length 2}. We start with the following lemma.

\begin{lemma}\label{computer checks}
We have 
        \begin{equation*}
            p_k^2(n)>p_k(n-1)p_k(n+1)
        \end{equation*}
        if either
         $k=3$ and $2\le n \le 2\cdot 3^{11}+1$,
         $k=4$ and $1 \le n \le 2\cdot 10^5$, or
         $k=5$ and $1 \le n \le 8\cdot 10^5$.
\end{lemma}
\begin{proof}
	The right-hand side of Theorem \ref{thm of BKRT} is always positive thanks to the assumption that $n>\ell+1$. Taking $\alpha=k \in \{3,4,5\}$, $n\mapsto n+1$, and $\ell=n-1$ gives, for $n\geq 2k^{11}+\frac{k}{24}+1$,
	\begin{equation}\label{eqn:bign}
	p_k^2(n)>p_k(n-1)p_k(n+1).
	\end{equation}
	So we are left to show that $p_k(n)$ is strictly log-concave for $k\in\{3,4,5\}$ and for $n\leq 2k^{11}+1$ (as $0<\frac{k}{24}<1$). The upper bound of $n$ for $k\in \{3,4,5\}$ were done. Using a computer (Apple MacBook Air M3, 8 GB RAM, and 256GB memory throughout).
\end{proof}
 \begin{remark}
 	By Lemma~\ref{reduce} and Lemma~\ref{computer checks}, we are left to verify Theorem \ref{thm:general} for the following cases:
 	\begin{enumerate}
 		\item $2\cdot 10^5+1 \le n \le 2\cdot 4^{11}+1$ for $k=4$,
 		
 		\item $8\cdot 10^5+1 \le n \le 2\cdot 5^{11}+1$ for $k=5$.
 	\end{enumerate}
 \end{remark}

 The remaining cases in the above remark seem to be too heavy to verify numerically. For the case of general log-concavity, the numerical difficulty was overcome in \cite[Section 4]{MR4287510} by using certain sequences to establish upper and lower bounds for $p_k(n)$ to tackle the remaining cases. We next recall these sequences. Let ${\bm d}=(d_j)_{j\ge1}$ be a sequence of positive integers with $d_j\leq j$, and for $n\in \mathbb{N}$ we recursively define $p_{k,{\bm d}}^{\pm}(0):=1$ and, for $n\in\N$,
    \begin{align*}
            p_{k,{\bm d}}^{-}(n)&:= \frac{k}{n}\sum_{\ell=1}^{d_n}\sigma(\ell)p_{k,{\bm d}}^-(n-\ell),\\
            p_{k,{\bm d}}^{+}(n)&:= \frac{k}{n}\sum_{\ell=1}^{d_n}\sigma(\ell)p_{k,{\bm d}}^+(n-\ell)+knp_{k,{\bm d}}^+(n-d_n-1),
    \end{align*}
    where $\sigma(\ell) := \sum_{d\mid\ell}d$ denotes the sum of divisors of $\ell$. We also set
    \begin{equation*}
        p_{k,{\bm d}}^{\pm}(n):=0 \text{ for } n\leq -1.
    \end{equation*}
Using these sequences, upper and lower bounds for $p_k(n)$ were established in \cite[Lemma 4.1]{MR4287510}.
\begin{lemma}\label{2}
    For $n\in \mathbb{N}$ we have
    \begin{equation*}
        p_{k,{\bm d}}^-(n)\leq p_k(n)\leq p_{k,{\bm d}}^+(n).
    \end{equation*}
\end{lemma}
    It turns out that the first inequality in Lemma \ref{2} is often strict. The following lemma is not hard to show.
    \begin{lemma}\label{3}
       For a sequence $\textbf{d}=(d_j)_{j\ge1}$ with $d_j<j$, we have $p_{k,{\bm d}}^-(n)< p_k(n)$.
    \end{lemma}

Following \cite[Proposition 4.3]{MR4287510}, for $k=4$, we define $\bm{d}=\bm{d}_{\bm4}$ by  
\begin{equation*}%\label{eqn:d4def}
    d_j = d_{4,j} := 
\begin{cases}
j & \text{if } j \leq 2 \cdot 10^5, \\[8pt]
\left\lfloor 250 \, j^{\frac{1}{3}} \right\rfloor & \text{if } 2 \cdot 10^5 < j \leq 3.5 \cdot 10^6, \\[8pt]
\left\lfloor 1125 \, j^{\frac{1}{3}} \right\rfloor & \text{if } j > 3.5 \cdot 10^6
\end{cases}
\end{equation*}
Moreover, for $k=5$, we define $\bm{d}=\bm{d}_{\bm5}$ by  
\begin{equation}\label{eqn:d5def}
d_j = d_{5,j} := 
\begin{cases}
j & \text{if } j \leq 8 \cdot 10^5, \\[8pt]
\left\lfloor 25 \, \sqrt{j} \right\rfloor & \text{if } 8 \cdot 10^5 < j \leq 2 \cdot 10^7, \\[8pt]
\left\lfloor \frac{43}{2} \, \sqrt{j} \right\rfloor & \text{if } j > 2 \cdot 10^7.
\end{cases}
\end{equation}
To prove Theorem \ref{st. log con for length 2}, we need the following lemma that is not hard to show.
\begin{lemma}\label{lem:dj<j}
\noindent

\noindent
\begin{enumerate}
\item If $j>2\cdot 10^5$, then we have $d_{4,j}<j$.
\item If $j>8\cdot 10^5$, then we have $d_{5,j}<j$.
\end{enumerate}
\end{lemma}

We are now ready to prove Theorem \ref{st. log con for length 2}. 
\begin{proof}[Proof of Theorem \ref{st. log con for length 2}]
We are left to prove Theorem \ref{st. log con for length 2} for the cases mentioned in the remark below Lemma \ref{computer checks}.

We first consider the case $k=4$. Note that $2\cdot 4^{11}+1<8.4\cdot 10^{6}$.
In the proof of \cite[Proposition 4.3]{MR4287510} it was shown that, for all $n \leq 8.4 \cdot 10^6$, we have
\begin{equation}\label{Verweis}
    \frac{p_{4,{\bm{d_4}}}^-(n)}{p_{4,{\bm {d_4}}}^+(n-1)}\geq\frac{p_{4,{\bm{d_4}}}^+(n+1)}{p_{4,{\bm {d_4}}}^-(n)}.
\end{equation}
Combining Lemma \ref{3} with Lemma \ref{lem:dj<j} (1), for $n>2\cdot 10^5$, we have $p_{4,\bm{d_4}}^{-}(n)<p_4(n)$. Using \eqref{Verweis} and Lemma \ref{2}, we obtain, for $2\cdot10^5<n\leq8.4\cdot 10^6$,
\begin{equation*}
    p_4^2(n)>p_{4,{\bm{d_4}}}^-(n)^2\geq p_{4,{\bm{d_4}}}^+(n-1)p_{4,{\bm{d_4}}}^+(n+1)\geq p_4(n-1)p_4(n+1).
\end{equation*}

We next consider the case $k=5$. Note that $2\cdot 5^{11}+1< 9.9 \cdot 10^7$. We use the sequence in \eqref{eqn:d5def}. In the proof of \cite[Proposition 4.3]{MR4287510} it was shown that, for $n \leq 9.9 \cdot 10^7$,
\begin{equation*}
    \frac{p_{5,{\bm{d_5}}}^-(n)}{p_{5,{\bm {d_5}}}^+(n-1)}\geq\frac{p_{5,{\bm{d_5}}}^+(n+1)}{p_{5,{\bm {d_5}}}^-(n)}.
\end{equation*}
Combining Lemma \ref{3} with Lemma \ref{lem:dj<j} (2), for $n>8\cdot 10^5$, we have $p_{5,\bm{d}}^{-}(n)<p_5(n)$.  Applying Lemma \ref{2}, we obtain, for $8\cdot 10^5<n\le 9.9 \cdot 10^7$, 
\begin{equation*}
    p_5^2(n)>p_5(n-1)p_5(n+1). 
\end{equation*}
Combining gives the claim.
\end{proof}
 
 We are now ready to prove Corollary \ref{st. log con for ln 2, equv}.

\begin{proof}[Proof of Corollary \ref{st. log con for ln 2, equv}]\hspace{0cm}
	
\noindent(1) By Theorem \ref{st. log con for length 2}, we have, for $n\in \mathbb{N}$ and $k\ge 4$, 
\begin{equation*}%\label{eqn:log conc for k> 3}
    p_k^2(n)>p_k(n-1)p_k(n+1).
\end{equation*}
In particular, this holds for $1\leq n\leq N-1$ (for any $N\in \mathbb{N}_{\ge 2}$). We now use Lemma \ref{eqv reln for st log conc.} with $r=1$, $a_n=p_k(n)$, $n_1=\ell$, and $n_2=n$. Then we have, for $1\leq \ell+1<n\le N$,
\begin{equation*}
    p_k(n-1)p_k(\ell+1)>p_k(n)p_k(\ell).
\end{equation*}
This yields (1). 

\noindent(2) For $n\ge 2$, we have, by Theorem \ref{st. log con for length 2}, 
\begin{equation*}%\label{eqn: log conc for k=3}
	p_3^2(n)>p_3(n-1)p_3(n+1).
\end{equation*}
This in particular holds for any $2\le n\le N-1$.
Again we use Lemma \ref{eqv reln for st log conc.} with $a_n=p_3(n)$, $r=2$, $n_1=\ell$, and $n_2=n$. Then, for all $2\leq \ell+1<n\le N$,
\begin{equation*}
    p_3(n-1)p_3(\ell+1)>p_3(n)p_3(\ell).
\end{equation*}
This completes the proof of (2).

\noindent(3) For $k=2$, \eqref{eqn:bign} implies the claim for $n>2^{12}+1$. We used a computer to verify the cases $n\leq 2^{12}+1$. This completes the proof of (3).
\end{proof}
\section{Proof of Theorem \ref{thm:general}}\label{sec:Inequality} 

To deal with the proof of Theorem \ref{thm:general}, we next recall an important notion that appears frequently in number theory and combinatorics. (see \cite[Subsection 1.2]{arnold2012majorization}, \cite[Definition 5.1]{doi:10.1142/S179383090900035X}). 
\begin{definition}\label{def: robin hood}
    For ${\bm n}=(n_1,n_2,\cdots,n_r)\in \mathbb{N}^r$, we define a \textit{Robin Hood transformation}\footnote{In the literature, the name ``Pigou--Dalton'' transfer is also used.} as a map $T\colon\mathbb{N}^r\to \mathbb{N}^r$, such that for some $1\leq j,\ell\leq r$ with $n_j>n_\ell$, the vector $T({\bm n}):={\bm N}=(N_1,N_2,\cdots,N_r)$ satisfies $N_j=n_j-1,N_\ell=n_\ell+1$, and $N_k=n_k$ for $k\not\in \{j,\ell\}$.
\end{definition}

\begin{remark}\label{rmk: robin hood induces majorization}
     By Definition \ref{def: robin hood}, ${\bm n}$ majorizes $T({\bm n})$.
\end{remark}

The following lemma is well-known in combinatorics. 
\begin{lemma}\label{3.2.16}
    Let ${\bm a}=(a_1,a_2,\cdots,a_r)$ and ${\bm b}=(b_1,b_2,\cdots,b_r)$ be partitions of $n\in\N_0$. Then  ${\bm b}\succ{\bm a}$ if and only if ${\bm a}$ is obtained from ${\bm b}$ by a finite number of Robin Hood transformations.
\end{lemma}
\begin{proof}
Using Muirhead's work \cite{doi:10.1017/S001309150003460X}, Hardy--Littlewood--P\'olya \cite[Subsection 2.19, Lemma 2]{hardy1952inequalities} proved the forward direction; also see \cite[Subsection 1.2]{arnold2012majorization}.

We next prove the reverse direction. From Definition \ref{Def: majorisation}, it is clear that majorization is a partial ordering, in particular, a transitive ordering in the set of all partitions of a positive integer. Thus, from the observation in Remark \ref{rmk: robin hood induces majorization} and the assumption that ${\bm a}$ is obtained from ${\bm b}$ by a finite number of transformations, we have ${\bm b} \succ {\bm a}$. Hence, our proof is complete.
\end{proof}
We are now ready to prove Theorem \ref{thm:general}.
\begin{proof}[Proof of Theorem \ref{thm:general}]
 We study how $p_k$ behaves for all possible images of a partition ${\bm b}=(b_1,b_2,\cdots,b_r)$ of $n$ under a Robin Hood transformation $T$.\\
\textbf{Case 1 : $b_j=b_\ell+1$.} In this case
\begin{equation*}
        \begin{split}
            T({\bm b})&=(b_1,b_2,\cdots,b_{j-1},b_j-1,b_{j+1},\cdots, b_{\ell-1},b_\ell+1,b_{\ell+1},\cdots,b_r)\\&=(b_1,b_2,\cdots,b_{j-1},b_\ell,b_{j+1},\cdots,b_{\ell-1},b_j,b_{\ell+1},\cdots,b_r).
        \end{split}
    \end{equation*}
Thus, $T({\bm b})$ is obtained from ${\bm b}$ by interchanging two components. Therefore 
 \begin{equation*}
        p_k({\bm b})=p_k(T({\bm b})).
    \end{equation*}
\textbf{Case 2 : $b_j>b_\ell+1$.} In this case
    \begin{equation*}
        \begin{split}
            T({\bm b})&=(b_1,\cdots,b_{j-1},b_j-1,b_{j+1},\cdots,b_{\ell-1},b_\ell+1,b_{\ell+1},\cdots,b_r).
        \end{split}
    \end{equation*}
Note that, $b_{\ell}\ge 1$, so $b_j>b_{\ell}+1\ge 2$. Thus, by Corollary \ref{st. log con for ln 2, equv} (1), (2), we have,
    \begin{align*}
            &p_k(T({\bm b}))\\
            &\!\!=p_k(b_1)\cdots p_k(b_{j-1})p_k(b_j-1)p_k(b_{j+1})\cdots p_k(b_{\ell-1})p_k(b_\ell+1)p_k(b_{\ell+1}) \cdots p_k(b_r)\\
            &\!\!>p_k(b_1)\cdots p_k(b_{j-1})p_k(b_j)p_k(b_{j+1})\cdots p_k(b_{\ell-1})p_k(b_\ell)p_k(b_{\ell+1}) \cdots p_k(b_r)=p_k({\bm b}).
    \end{align*}
    To prove Theorem \ref{eqn: gen. ineq.}, by \eqref{pv}, we can assume that $a_j\ne b_j$ for $j\in \{1,2,\cdots, r\}$. Therefore, without loss of generality, $a_1<b_1$. 
    
    \noindent Since ${\bm b}\succ{\bm a}$, by Lemma \ref{3.2.16}, there exist Robin Hood transformations $T_1, T_2$, $\cdots , T_s$ such that 
    \begin{equation*}
        T_s\circ T_{s-1}\cdots \circ T_1(\bm{b})=\bm{a}.
    \end{equation*}
    For $1\le m \le s$, we let
    \begin{align*}
        \bm{b}^{[m]}=\left(b^{[m]}_1,b^{[m]}_2,\cdots,b^{[m]}_r\right)&:= T_{m}\circ T_{m-1}\cdots \circ T_2\circ T_1(\bm{b}),
       \\ \left(b^{[0]}_1, b^{[0]}_2, \cdots, b^{[0]}_r\right)&:= (b_1, b_2,\cdots , b_r). 
    \end{align*}
      Since $\bm{a}\ne \bm{b}$, there exists at least one Robin Hood transformation $T_m$ with $m\in \{1,2,\cdots,s\}$, such that, for some $1\le \ell, j \le r$, we have $b^{[m-1]}_j>b^{[m-1]}_{\ell}+1$ (see Definition \ref{def: robin hood}). Indeed if $b^{[m-1]}_j=b^{[m-1]}_{\ell}+1$ for every $1\le m \le s$, then we would have
     \begin{equation*}
         \bm{b}=\bm{b}^{[1]}=\bm{b}^{[2]}\cdots=\bm{b}^{[s-1]}=\bm{a}.
     \end{equation*}
     This would be a contradiction. Therefore, by Case 2, we have
     \begin{equation*}
         \begin{split}
             p_k(\bm{b})\le p_k\left(\bm{b}^{[1]}\right)\le \cdots \le p_k\left(\bm{b}^{[m-1]}\right) < p_k\left(\bm{b}^{[m]}\right)\le \cdots \le p_k(\bm{a}).
         \end{split}
     \end{equation*}
     This finishes the proof.
\end{proof}

\section{Questions for future research}\label{sec:furtherquestions}

Since majorization only induces a partial ordering $\prec$, it is natural to consider inequalities between $p_k({\bm a})$ and $p_k({\bm b})$, where neither $\bm{a}$ nor $\bm{b}$ majorizes the other. For $n,r,R\in\N$ with $R\leq r-1$ and sequences $\bm{c}=(c_1,\dots,c_r)$ and $\bm{d}=(d_1,\dots,d_r)$ with exactly $r$ parts that sum to $n$ (i.e., $\sum_{j=1}^{r} c_r=\sum_{j=1}^r d_r=n$), consider the ``partial (strict) majorization'' 
    \begin{equation}\label{eqn:partialmajor}
           \hspace{-.3cm} \sum_{j=1}^{\ell} c_j\leq \sum_{j=1}^{\ell} d_j \text{ for } 1 \leq \ell \leq R \text{ and } \sum_{j=1}^{\ell} c_j<\sum_{j=1}^{\ell}d_j\text{ for some }1\leq \ell\leq R.
    \end{equation}
Noting that $p_k$ reverses the order from $\succ$ in Theorem \ref{thm:general}, computer calculations were run to check whether weaker conditions led to reversed ordering as well. These indicated that the ordering is sensitive to the smallest parts of the partition. Hence, for a partition $\bm{a}=(a_1,\dots,a_r)$ with $a_1\geq a_1\geq\dots \geq a_r$, we set\footnote{Although the non-increasing ordering for partitions used in this paper is more common, some authors define partitions with this reversed ordering on the parts.} ${\bm{a}^*}:=(a_r,a_{r-1},\dots,a_{1})$. Set
\begin{multline*}
\mathcal{S}_{n,r,R}^{*}:=\Big\{(\bm{a},\bm{b}): \bm{a}=(a_1,\dots,a_r), \bm{b}=(b_1,\dots, b_r)\text{ are partitions of $n$}\\[-.35em]
\text{with $\bm{c}={\bm{b}^*}$ and $\bm{d}={\bm{a}^*}$ satisfying \eqref{eqn:partialmajor}}\Big\}.
\end{multline*}
By Lemma \ref{3.2.16} it is straightforward to see that we have $\bm{b}\succ\bm{a} \text{ iff } (\bm{a},\bm{b})\in \mathcal{S}_{n,r,r-1}^{*}$. Thus, if $(\bm{a},\bm{b})\in \mathcal{S}_{n,r,r-1}^{*}$, then we have $p_{k}(\bm{a})>p_k(\bm{b})$ for $k\geq 3$ by Theorem \ref{thm:general}. Hence, for $r\in\N$ and $k\geq 3$, there exists $R_{r,k}\leq r$ minimal such that for $(\bm{a},\bm{b})\in \mathcal{S}_{n,r,R_{r,k}}^{*}$, we have
\[
p_{k}(\bm{a})>p_k(\bm{b}).
\]

In another direction, one can fix $R$ and ask which $(\bm{a},\bm{b})\in \mathcal{S}_{n,r,R}^{*}$ and $k\in\N$ satisfy the inequality $p_k(\bm{a})<p_k(\bm{b})$. Using a computer, we collect some data for $R=1$ (with $r\leq 10$ and $n\leq 50$) in Table \ref{table:1}. In this case, \eqref{eqn:partialmajor} for $\bm{c}={\bm{b}^*}$ and $\bm{d}={\bm{a}^*}$ becomes $a_r> b_r$. 
%If $a_1=b_1$, then $p_k(\bm{a})<p_k(\bm{b})$ if and only if 
%\[
%p_k(a_2,\dots,a_{r})<p_k(b_2,\dots,b_{r}),
%\]
% so we can remove this part and reduce the question to a question about partitions of $n-a_1$ with precisely $r-1$ parts. For $\bm{a}\neq \bm{b}$, after removing enough parts we eventually can assume without loss of generality that $a_1<b_1$.

For $r, n\in\N$, let $\mathcal{S}_{n,r}^{*}$ denote the set of pairs $(\bm{a},\bm{b})$ of partitions of $n$ with exactly $r$ parts satisfying $a_r>b_r$. For $k\in\N$, we then define
\begin{align*}
\mathcal{S}_{k,n,r}^{*,<}&:=\left\{(\bm{a},\bm{b})\in \mathcal{S}_{n,r}^{*}:p_k(\bm{a})<p_k(\bm{b})\right\},\\
\mathcal{S}_{k,n,r}^{*,=}&:=\left\{(\bm{a},\bm{b})\in \mathcal{S}_{n,r}^{*}:p_k(\bm{a})=p_k(\bm{b})\right\},\\
\mathcal{S}_{k,n,r}^{*,>}&:=\left\{(\bm{a},\bm{b})\in \mathcal{S}_{n,r}^{*}:p_k(\bm{a})>p_k(\bm{b})\right\}.
\end{align*}
Then we can study the quantities 
\[
\frac{\#\mathcal{S}_{k,n,r}^{*,<}}{\#\mathcal{S}_{n,r}^{*}},\quad\frac{\#\mathcal{S}_{k,n,r}^{*,=}}{\#\mathcal{S}_{n,r}^{*}},\quad \frac{\#\mathcal{S}_{k,n,r}^{*,>}}{\#\mathcal{S}_{n,r}^{*}}.
\]

Computer calculations indicate that $\mathcal{S}_{k,n,r}^{*,=}=\emptyset$ for $k\geq 4$ and that $\frac{\#\mathcal{S}_{k,n,r}^{*,>}}{\#\mathcal{S}_{n,r}^{*}}$ is much larger than $\frac{\#\mathcal{S}_{k,n,r}^{*,<}}{\#\mathcal{S}_{n,r}^{*}}$. We give some indicative data in the table below.
\begin{longtable}{|c|c|c|c|c|c|c|c|}
\caption{Some output data}
\label{table:1}\\
\hline
$n$&$r$&$k$&$\#\mathcal{S}_{n,r}^{*}$& $\#\mathcal{S}_{k,n,r}^{*,>}$& $\frac{\#\mathcal{S}_{k,n,r}^{*,>}}{\#\mathcal{S}_{n,r}^{*}}$&
$\#\mathcal{S}_{k,n,r}^{*,<}$& $\frac{\#\mathcal{S}_{k,n,r}^{*,<}}{\#\mathcal{S}_{n,r}^{*}}$\rule{0cm}{.57cm}.\\[.2em]
\hline
$13$&$3$&$4$& $91$& $87$& $\frac{87}{91}\approx 95.6\%$& $4$& $\frac{4}{91}\approx 4.4\%$\rule{0pt}{2.4ex}\\[.2em]
$17$&$3$&$5$& $276$& $262$& $\frac{262}{276}\approx 94.9\%$& $12$& $\frac{12}{276}\approx 5.1\%$\\[.2em]
$20$&$3$&$5$& $528$& $495$& $\frac{495}{528}=93.75\%$& $33$& $\frac{33}{528}=6.25\%$\\[.2em]
$20$&$4$&$4$& $2016$& $1841$& $\frac{1841}{2016}\approx 91.32\%$& $175$& $\frac{175}{2016}\approx 8.68\%$\\[.2em]
$30$&$3$&$3$& $2775$& $2566$& $\frac{2566}{2775}\approx 92.47\%$& $209$& $\frac{209}{2775}\approx 7.53\%$\\[.2em]
$30$&$3$&$4$& $2775$& $2567$& $\frac{2567}{2775}\approx 92.5\%$& $208$& $\frac{208}{2775}\approx 7.5\%$\\[.2em]
$30$&$3$&$5$& $2775$& $2566$& $\frac{2566}{2775}\approx 92.47\%$& $209$& $\frac{209}{2775}\approx 7.53\%$\\[.2em]
$30$&$3$&$6$& $2775$& $2565$& $\frac{2565}{2775}\approx 92.43\%$& $210$& $\frac{210}{2775}\approx 7.57\%$\\[.2em]
$35$&$3$&$4$& $5151$& $4724$& $\frac{4724}{5151}\approx 91.71\%$& $427$& $\frac{427}{5151}\approx 8.29\%$\\[.2em]
$35$&$3$&$5$& $5151$& $4722$& $\frac{4722}{5151}\approx 91.67\%$& $429$& $\frac{429}{5151}\approx 8.33\%$\\[.2em]
$45$&$3$&$8$& $14196$& $12943$& $\frac{12943}{14196}\approx 91.17\%$& $1253$& $\frac{1253}{14196}\approx 8.83\%$\\[.2em]
$45$&$4$&$5$& $225456$& $194593$& $\frac{194593}{225456}\approx 86.31\%$& $30863$& $\frac{30863}{225456}\approx 13.69\%$\\[.2em]
$45$&$4$&$6$& $225456$& $194571$& $\frac{194571}{225456}\approx 86.3\%$& $30885$& $\frac{30885}{225456}\approx 13.7\%$\\[.2em]
$45$&$4$&$7$& $225456$& $194539$& $\frac{194539}{225456}\approx 86.29\%$& $30917$& $\frac{30917}{225456}\approx 13.71\%$\\[.2em]
$45$&$4$&$8$& $225456$& $194508$& $\frac{194508}{225456}\approx 86.27\%$& $30948$& $\frac{30948}{225456}\approx 13.73\%$\\[.2em]
$45$&$4$&$9$& $225456$& $194466$& $\frac{194466}{225456}\approx 86.25\%$& $30990$& $\frac{30990}{225456}\approx 13.75\%$\\[.2em]
$45$&$4$&$10$& $225456$& $194425$& $\frac{194425}{225456}\approx 86.24\%$& $31031$& $\frac{31031}{225456}\approx 13.76\%$\\[.2em]
\hline

\end{longtable}

As noted above, if $(\bm{a},\bm{b})\in \mathcal{S}_{n,r,r-1}^{*}$ and $k\geq 3$, then $(\bm{a},\bm{b})\in \mathcal{S}_{k,n,r}^{*,>}$, so it is reasonable to expect that $\frac{\#\mathcal{S}_{k,n,r}^{*,>}}{\#\mathcal{S}_{n,r}^{*}}$ is larger than $\frac{\#\mathcal{S}_{k,n,r}^{*,>}}{\#\mathcal{S}_{n,r}^{*}}$. 
In view of these considerations, it is natural to ask the following questions.
\begin{enumerate}
\item Do there exist $r$ and $k$ for which $R_{r,k}<r-1$?
\item If $R_{r,k}<r-1$ for some $r$ and $k$, then how does $R_{r,k}$ grow, as a function of $r$ or $k$? 
\item For fixed $r$, can one evaluate 
\begin{equation}\label{eqn:minuslimit}
\lim_{n\to\infty} \frac{\#\mathcal{S}_{k,n,r}^{*,>}}{\#\mathcal{S}_{n,r}^{*}}
\end{equation}
and/or 
\begin{equation}\label{eqn:pluslimit}
\lim_{n\to\infty} \frac{\#\mathcal{S}_{k,n,r}^{*,<}}{\#\mathcal{S}_{n,r}^{*}}?
\end{equation}
\item Is the limit \eqref{eqn:minuslimit} (resp. \eqref{eqn:pluslimit}) equal to $1$ (resp. $0$) for $k$ sufficiently large? 
\item Can one prove that $\mathcal{S}_{k,n,r}^{*,=}=\emptyset$ for $k\geq 4$ and/or prove the weaker claim that 
\[
\lim_{n\to\infty} \frac{\#\mathcal{S}_{k,n,r}^{*,=}}{\#\mathcal{S}_{n,r}^{*}}=0?
\]
\end{enumerate}

\end{document}